\documentclass[11pt]{amsart}
\usepackage{amssymb,amsmath,mathrsfs}
\usepackage{enumerate,tikz}
\usepackage{hyperref}
\usepackage{youngtab}

\usetikzlibrary{patterns}
\usetikzlibrary{decorations.pathreplacing, calligraphy}

\newtheorem{thm}{Theorem}[section]
\newtheorem{prop}[thm]{Proposition}
\newtheorem{cor}[thm]{Corollary}
\newtheorem{lem}[thm]{Lemma}
\newtheorem*{conjecture}{Conjecture}

\theoremstyle{definition}
\newtheorem*{remark}{Remark}

\numberwithin{equation}{section}

\def\Gr{\mathscr{G}}
\def\A{\mathcal{A}}
\def\E{\mathcal{E}}
\def\U{\mathsf{U}}
\def\D{\mathsf{D}}
\def\H{\mathsf{H}}
\def\Dyck{\textup{Dyck}}
\def\Schr{\textup{Schr}}

\def\abs#1{\lvert#1\rvert}
\DeclareMathOperator{\bin}{bin}
\DeclareMathOperator{\des}{des}
\DeclareMathOperator{\inv}{inv}
\DeclareMathOperator{\id}{id}
\DeclareMathOperator{\val}{val}

\def\oeis#1{\cite[#1]{Sloane}}

\def\R{\rule[-1ex]{0ex}{3.6ex}}

\tikzstyle{mesh}=[pattern=north east lines, pattern color=gray, draw=gray]
\tikzstyle{meshPatt}=[pattern=north west lines, pattern color=black!30!green, draw=gray]
\tikzstyle{meshMost}=[pattern=crosshatch, pattern color=black!10!orange!70, draw=gray]

\newcommand\dyckpath[3]{
\def\diam{0.08}
  \begin{scope}
    \draw[help lines] (#1) -- ++(#2*2,0);
    \draw[line width=1pt] (#1) foreach \dir in {#3}{ -- ++(\dir*90-45:1.41)};
    \draw[fill] (#1) circle (\diam);
    \draw[fill] (#1) foreach \dir in {#3}{ ++(\dir*90-45:1.41) circle (\diam)};
  \end{scope}
}

\begin{document}
\title{Restricted Grassmannian permutations}
\author[Juan Gil]{Juan B. Gil}
\address{Penn State Altoona\\ 3000 Ivyside Park\\ Altoona, PA 16601}
\email{jgil@psu.edu}

\author[Jessica Tomasko]{Jessica A. Tomasko}
\address{Penn State Altoona\\ 3000 Ivyside Park\\ Altoona, PA 16601}
\email{jat5880@psu.edu}

\begin{abstract}
A permutation is called Grassmannian if it has at most one descent. In this paper, we investigate pattern avoidance and parity restrictions for such permutations. As our main result, we derive formulas for the enumeration of Grassmannian permutations that avoid a classical pattern of arbitrary size. In addition, for patterns of the form $k12\cdots(k-1)$ and $23\cdots k1$, we provide combinatorial interpretations in terms of Dyck paths, and for $35124$-avoiding Grassmannian permutations, we give an explicit bijection to certain pattern-avoiding Schr\"oder paths. Finally, we enumerate the subsets of odd and even permutations and discuss properties of their corresponding Dyck paths.   
\end{abstract}

\maketitle

\section{Introduction}

A {\em Grassmannian permutation} is a permutation having at most one descent. They are a special type of vexillary permutation ($2143$-avoiding) that is associated to Schubert varieties, cf.\ \cite{LS85, Lenart00, Mac91}. If $\Gr_n$ denotes the set of Grassmannian permutations on $[n]=\{1,\dots,n\}$, then $\pi\in\Gr_n$ if and only if $\pi^{rc}\in\Gr_n$, where $\pi^{rc}$ is the reverse complement of $\pi$. We let $\Gr_n(\sigma)$ be the subset of Grassmannian permutations on $[n]$ that avoid the pattern $\sigma$.

Every subset $A$ of $[n]$ gives rise to an element of $\Gr_n$ obtained by listing the elements of $A$ in increasing order, followed by the elements of the complement of $A$, also listed in increasing order. This construction gives $2^n$ permutations, including $n+1$ copies of the identity permutation $\id_n=1\cdots n$. Thus,
\[ \abs{\Gr_n}=2^n-n \;\text{ for } n\ge 1. \]
The first few terms are $1, 2, 5, 12, 27, 58, 121, 248, 503, 1014,\dots$. See \oeis{A000325} for other combinatorial interpretations of this sequence. Moreover, if $n>1$ and $k\in[n-1]$, then every $\pi\in\Gr_n$ with $\pi(k)>\pi(k+1)$ must be of the form 
\begin{gather*}
 \pi = \pi(1)\cdots \pi(k) \,|\, 1\, \pi(k+2)\cdots\pi(n) \;\text{ or }\\
  \pi = 1\cdots j\, \pi(j+1)\cdots \pi(k) \,|\, (j+1) \pi(k+2) \cdots \pi(n) \text{ for some } j\le k-1,
\end{gather*} 
where the symbol $|$ is just used to indicate where the descent of $\pi$ happens. Therefore, there are $\sum_{j=0}^{k-1} \binom{n-j-1}{k-j}$ permutations in $\Gr_n$ having their descent at position $k$.

Equivalently, Grassmannian permutations can be defined using their Lehmer code. Recall that for $\pi\in S_n$, the code of $\pi$ is the word $L(\pi) = c_1(\pi)\cdots c_n(\pi)$, where
\[ c_i(\pi) = \#\{j: j>i \text{ and } \pi(j)<\pi(i)\}. \]
Note that $c_i(\pi)>c_{i+1}(\pi)$ if and only if $\pi(i)>\pi(i+1)$, so $\pi\in\Gr_n$ if and only if there is a $k$ such that $c_1(\pi)\le\cdots\le c_{k}(\pi)$ and $c_i(\pi)=0$ for $i > k$.

The goal of this paper is to study the enumeration of pattern-avoiding Grassmannian permutations as well as other related subsets of $\Gr_n$, including involutions, biGrassmannian permutations (i.e.\ $\Gr_n \cap\Gr_n^{-1}$), and permutations with an even or odd number of inversions. Most of the arguments presented here are combinatorial, and for the most part, we will use the standard notation found in the book by Kitaev~\cite{Kitaev11}. 

In Section~\ref{sec:GrassPerms}, we review some basic (partly known) results regarding the sets $\Gr_n \cap \Gr_n^{-1}$ and $\Gr_n \cup \Gr_n^{-1}$, and discuss the structure and enumeration of Grassmannian involutions.

Our main results are presented in Section~\ref{sec:avoidPattern}, where we discuss pattern avoidance for patterns of arbitrary size. It turns out that the cardinality of $\Gr_n(\sigma)$ only depends on $\des(\sigma)$, the number of descents of $\sigma$. More precisely, $\Gr_n(\sigma)=\Gr_n$ if $\des(\sigma)>1$, $\abs{\Gr_n(\sigma)}$ is finite if $\des(\sigma)=0$, and if $|\sigma|=k$ and $\des(\sigma)=1$, then $\abs{\Gr_n(\sigma)}  = 1 + \sum_{j=3}^k\binom{n}{j-1}$. For the finite class $\Gr_n(1\cdots k)$, we discuss some formulas and prove that $\abs{\Gr_{2k-2}(12\cdots k)} = \frac{1}{k}\binom{2k-2}{k-1} = C_{k-1}$, giving yet another (possibly new) interpretation of the Catalan numbers.

In Section~\ref{sec:bijections}, we give a bijective map between $\Gr_n$ and Dyck paths of semilength $n$ having at most one long ascent\footnote{A Dyck path is a lattice path from $(0,0)$ to $(2n,0)$ using up-steps $\U=(1,1)$, down-steps $\D=(1,-1)$, never going below the $x$-axis. A long ascent is a sequence of two or more up-steps.}, and use that map to provide combinatorial interpretations for $\Gr_n(k12\cdots(k-1))$ and $\Gr_n(23\cdots k1)$. In addition, we use the Lehmer code to establish a connection between the elements of $\Gr_n(35124)$ and certain pattern-avoiding Schr\"oder paths discussed in \cite{CioFer17}.

Finally, in Section~\ref{sec:evenoddPerms}, we discuss odd and even Grassmannian permutations. For these families, we derive enumeration formulas and provide combinatorial interpretations in terms of Dyck paths. The study of pattern avoidance for parity restricted permutations is more delicate and will be pursued in future work.

\section{Grassmannian related permutations}
\label{sec:GrassPerms}

A permutation $\pi$ is called {\em biGrassmannian} if $\pi$ and $\pi^{-1}$ are both Grassmannian. Observe that if $\pi$ is Grassmannian then $\pi^{-1}$ has at most one {\em dip}, that is a pair $(i,j)$ with $i<j$ such that $\pi(i)=\pi(j)+1$. In other words, a permutation is biGrassmannian if and only if it has at most one descent and at most one dip. For example, $\pi=2413\in\Gr_4$ has two dips and is the only Grassmannian permutation of size 4 that is not biGrassmannian.

\begin{prop}\label{prop:biGrass}
A Grassmannian permutation is biGrassmannian if and only if it avoids the pattern $2413$. In other words, $\Gr_n \cap \Gr_n^{-1} = \Gr_n(2413)$ for every $n\in\mathbb{N}$. Moreover,
\[ \abs{\Gr_n \cap \Gr_n^{-1}} = 1 + \binom{n+1}{3}. \]
\end{prop}
\begin{proof}
Since $3142$ has two descents, we have $\Gr_n\subset S_n(3142)$. Thus $\Gr_n^{-1}\subset S_n(2413)$ and therefore, $\Gr_n\cap \Gr_n^{-1}\subset \Gr_n(2413)$. Conversely, suppose $\pi$ is Grassmannian and has two dips, say $(i_1,j_1)$ and $(i_2,j_2)$ with $i_1<i_2$. Since $\pi$ has at most one descent and avoids a $321$ pattern, we must have
\[ i_1< i_2 < j_1 < j_2 \;\text{ and }\; \pi(j_1) < \pi(i_1)< \pi(j_2) < \pi(i_2), \] 
which gives a $2413$ pattern. In other words, $\Gr_n(2413)\subset \Gr_n^{-1}$ and so $\Gr_n(2413)\subset \Gr_n\cap \Gr_n^{-1}$. In conclusion, $\Gr_n \cap \Gr_n^{-1} = \Gr_n(2413)$.

For the enumeration of $\Gr_n \cap \Gr_n^{-1}$, we organize the permutations by the position of the descent (if any). There is only one permutation with no descent: $\pi = 1\cdots n$. Moreover, if $\pi\in \Gr_n \cap \Gr_n^{-1}$ starts with a descent, then it must be of the form 
\[ \pi = n\,1\cdots(n-1) \;\text{ or }\; \pi = i\,1\cdots(i-1)(i+1)\cdots n \;\text{ with } i\in\{2,\dots,n-1\}. \]
There are $\binom{n-1}{1}$ such permutations. 

Now, if $\pi$ has only one dip and its descent is at a position greater than 1, then its initial ascending run must consist of at most two consecutive runs with a gap to account for the dip. In other words, $\pi$ can only take one of the following forms:
\begin{enumerate}[a)]
\itemindent20pt
\item $\pi= \diamond\, i\cdots j \,|\, 1\,\tau_1$ \ with $i,j\in \{2,\dots,n\}$ and $i<j$;
\item $\pi=1\cdots (i-1)\diamond j \,|\, i\,\tau_2$ \ with $i,j\in \{2,\dots,n\}$ and $i<j$;
\item $\pi=1\cdots (i-1)\diamond j\cdots k \,|\, i\,\tau_3$ \ with $i,j,k\in \{2,\dots,n\}$ and $i<j<k$;
\end{enumerate}
where $\diamond$ indicates the gap in the first ascending run of $\pi$, and $\tau_1, \tau_2, \tau_3$ are increasing permutations. Cases a) and b) give $\binom{n-1}{2}$ permutations each, and case c) gives $\binom{n-1}{3}$ permutations with the desired properties.

In conclusion, there is a total of
\[ 1 + \binom{n-1}{1} + 2\binom{n-1}{2} + \binom{n-1}{3} = 1 + \binom{n}{2} +  \binom{n}{3} = 1 + \binom{n+1}{3}\]
biGrassmannian permutations of size $n$.
\end{proof}

Next, we consider the set $\Gr_n \cup \Gr_n^{-1}$ of Grassmannian permutations together with their inverses. This set can be described in terms of permutations avoiding a pair of patterns.
\begin{prop}
For $n\in\mathbb{N}$, we have $\Gr_n \cup \Gr_n^{-1} = S_n(321,2143)$. Moreover, 
\[ \abs{\Gr_n \cup \Gr_n^{-1}} = 2^{n+1}-\binom{n+1}{3}-2n-1. \]
\end{prop}
\begin{proof}
Clearly, $\Gr_n\subset S_n(321,2143)$ and so $\Gr_n^{-1}\subset S_n(321,2143)$ (note that $321$ and $2143$ are self-inverses). Hence $\Gr_n\cup\Gr_n^{-1}\subset S_n(321,2143)$. On the other hand, if $\pi\in S_n(321,2143)$ has more than one descent, then the plot of $\pi$ must be of the form

\medskip
\begin{center}
\begin{tikzpicture}[scale=0.6]
\node[below=1] at (1,0.25) {$i$};
\node[below=1] at (4,0.25) {$j$};
\clip (0.2,0.2) rectangle (4.8,4.8);
\draw[gray] (0,0) grid (5,5);
\draw[mesh] (1,0) rectangle (2,5); 
\draw[mesh] (3,0) rectangle (4,5);
\draw[meshPatt] (0,3) rectangle (1,5);
\draw[meshPatt] (4,0) rectangle (5,2);
\foreach \x/\y in {0/1,2/0,2/2,2/4,4/3}{\draw[meshPatt] (\x,\y) rectangle (1+\x,1+\y);}
\foreach \x/\y in {1/3,2/1,3/4,4/2}{\draw[fill] (\x,\y) circle (0.12);}
\end{tikzpicture}
\\
\tikz[scale=0.5]{
\draw[meshPatt] (0,0) rectangle (0.5,0.5);
\node at (3.15,0.25) {\scriptsize Pattern avoidance};
}
\hspace{1ex}
\tikz[scale=0.5]{
\draw[mesh] (0,0) rectangle (0.5,0.5);
\node at (1.8,0.25) {\scriptsize Descent};
}
\end{center}
where $\pi(i)$ is the top of the most-left descent and $\pi(j)$ is the bottom of the most-right descent of $\pi$. There cannot be any $21$ pattern entirely contained in any of the unshaded regions, and every element to the left of $\pi(i)$ and greater than $\pi(j)$ must be larger than any element to the right of $\pi(j)$ and less than $\pi(i)$. Therefore, $\pi$ has only one dip, hence $\pi\in \Gr_n^{-1}$. In conclusion, $S_n(321,2143) \subset \Gr_n \cup \Gr_n^{-1}$.

Finally, $\abs{\Gr_n}=\abs{\Gr_n^{-1}}$ implies $\abs{\Gr_n \cup \Gr_n^{-1}} = 2\abs{\Gr_n} - \abs{\Gr_n \cap \Gr_n^{-1}} = 2(2^n-n) - 1 - \binom{n+1}{3}$.
\end{proof}
The sequence $\big(\abs{\Gr_n \cup \Gr_n^{-1}}\big)_{n\in\mathbb{N}}$ starts with $1, 2, 5, 13, 33, 80, 185, 411, 885, 1862,\dots$; see \oeis{A088921} for other combinatorial interpretations.

\medskip
We end this section with a description and enumeration of Grassmannian involutions. Before we proceed, we recall the definitions of skew and direct sums of two permutations. The {\em skew sum}, denoted by $\pi_1\ominus\pi_2$, is the permutation obtained by listing the elements of $\pi_1$, each increased by the size of $\pi_2$, followed by the elements of $\pi_2$. The {\em direct sum},  $\pi_1\oplus\pi_2$, is the permutation obtained by listing the elements of $\pi_1$, followed by the elements of $\pi_2$ increased by the size of $\pi_1$.

\begin{prop}
A permutation $\pi\in\Gr_n$ is an involution if and only if it is of the form
\begin{equation*} 
 \pi = \id_{k_1} \oplus \big(\id_{k_2} \ominus \id_{k_2}\big) \oplus \id_{k_3} 
\end{equation*}
for some $k_1,k_2,k_3\in\mathbb{N}\cup\{0\}$ with $k_1+2k_2+k_3=n$, where $\id_0=\varepsilon$ is the empty word. Moreover, if $i_n$ denotes the number of Grassmannian involutions of size $n$, then
\begin{equation*}
 i_n = \begin{cases}
 \frac{n^2+3}{4} &\text{if $n$ is odd,}\\[3pt]
 \frac{n^2+4}{4} &\text{if $n$ is even.}
 \end{cases}
\end{equation*}
This sequence starts with $1, 2, 3, 5, 7, 10, 13, 17, 21, 26, 31, \dots$.
\end{prop}

\begin{proof}
First, it is clear that permutations of the claimed form are Grassmannian involutions. Suppose now that $\pi$ is an involution in $\Gr_n$. If $\pi$ is the identity, we can choose $k_1=n$ and $k_2=k_3=0$. Otherwise, we can decompose $\pi = \id_{k_1}\oplus\, \tau\oplus \id_{k_3}$ with $k_1,k_3\ge 0$, where $\tau$ is a Grassmannian involution of size $m\le n$ that does not start with 1 or end with $m$. Thus, $\tau$ consists of an increasing sequence ending with $m$ followed by an increasing sequence starting with $1$. Moreover, since $\tau\in \Gr_m \cap \Gr_m^{-1}=\Gr_m(2413)$, it must be of the form $\tau=\id_\ell\ominus\id_r$, and since $\tau$ is an involution, we must have $\ell=r$.

To enumerate the involutions in $\Gr_n$, we group them by the size of their indecomposable component $\tau$. For instance, there are $n-1$ involutions with $\tau=1\ominus 1$, $n-3$ involutions with $\tau=12\ominus 12$, and generally, $n-(2k-1)$ involutions with $\tau=\id_{k} \ominus \id_{k}$. We also have the identity (if $\tau=\varepsilon$). In conclusion, if $n=2m$ or $n=2m+1$, then we get a total of
\[ i_n = 1+\sum_{k=1}^m (n-(2k-1)) = 1 + mn - m^2\]
Grassmannian involutions, leading to the claimed formulas.
\end{proof}

\begin{remark}
Grassmannian involutions correspond to standard Young tableaux (SYT) having at most two rows, and whose second row consists of consecutive numbers. More precisely, if $k_2\not=0$, the permutation $ \id_{k_1} \oplus \big(\id_{k_2} \ominus \id_{k_2}\big) \oplus \id_{k_3}$ corresponds to the SYT of shape $(n-k_2,k_2)$ whose second row consists of the consecutive labels $k_1+k_2+1,\dots,k_1+2k_2$.

For example, the 10 Grassmannian involutions of size 6 give:

\bigskip
\begin{center}
\def\decom#1{\scriptsize\color{blue}$#1$}
\small
\begin{tabular}{ccccc}
\young(13456,2) & \young(12456,3) & \young(12356,4) & \young(12346,5) & \young(12345,6) \\
\decom{(1\ominus 1)\oplus 1234} & \decom{1\oplus(1\ominus 1)\oplus 123} & \decom{12\oplus(1\ominus 1)\oplus 12} & \decom{123\oplus(1\ominus 1)\oplus 1} & \decom{1234\oplus(1\ominus 1)} \\ 
2 1 3 4 5 6 & 1 3 2 4 5 6 & 1 2 4 3 5 6 & 1 2 3 5 4 6 & 1 2 3 4 6 5 \\[15pt]
\young(1256,34) & \young(1236,45) & \young(1234,56) & \young(123,456) & \young(123456) \\
\decom{(12\ominus 12)\oplus 12} & \decom{1\oplus(12\ominus 12)\oplus 1} & \decom{12\oplus(12\ominus 12)} & \decom{123\ominus 123} & \decom{123456} \\
 3 4 1 2 5 6 & 1 4 5 2 3 6 & 1 2 5 6 3 4 &  4 5 6 1 2 3 &  1 2 3 4 5 6
\end{tabular}
\end{center}
\end{remark}

\section{Pattern-avoiding Grassmannian permutations}
\label{sec:avoidPattern}

First of all, note that since Grassmanian permutations have at most one descent, they naturally avoid any pattern with more than one descent. Thus 
\[ \Gr_n = \Gr_n(\sigma) \;\text{ for every $\sigma$ with } \des(\sigma)>1. \]
Moreover, it can be easily checked that $\Gr_n(12\cdots k)=\varnothing$ if $n\ge 2k-1$. Thus $\abs{\Gr_n(12\cdots k)}$ gives a sequence with $2k-2$ nonzero elements. For example, $\abs{\Gr_n(123)}$ gives $1, 2, 4, 2, 0,\dots$, and $\abs{\Gr_n(1234)}$ gives $1, 2, 5, 11, 10, 5, 0,\dots$.  

Clearly,  $\abs{\Gr_{m}(12\cdots k)}=2^m-m$ for $m<k$, and $\abs{\Gr_{k}(12\cdots k)} = 2^k - k - 1$. On the other end, we have the appearance of the Catalan number $C_{k-1} = \frac{1}{k}\binom{2k-2}{k-1}$.

\begin{prop}
For $k\ge 2$, $\abs{\Gr_{2k-3}(12\cdots k)} = 2C_{k-1}$ and $\abs{\Gr_{2k-2}(12\cdots k)} = C_{k-1}$.
\end{prop}
\begin{proof}
Every $\pi\in\Gr_{2k-2}(12\cdots k)$ is of the form $\pi = \tau_1\,(2k-2)\,1\,\tau_2$, where $\tau_1$ and $\tau_2$ are increasing permutations with $\abs{\tau_1}=\abs{\tau_2}=k-2$. In addition, for each $i\in\{1,\dots,k-1\}$ there are at most $k-1-i$ elements greater than $\pi(i)$ and so at least $k-1-(k-1-i)=i$ inversions involving $i$. Therefore, the Lehmer code of $\pi$ must be of the form
\[ L(\pi) = c_1\cdots c_{k-1}\, 0^{k-1} \text{ with } c_{k-1}=k-1 \text{ and } i\le c_i\le c_j \text{ for } i<j. \]
The formula $\abs{\Gr_{2k-2}(12\cdots k)} = C_{k-1}$ follows now from the fact that these codes are in one-to-one correspondence with Dyck paths of semilength $k-1$ via the bijective map
\[ c_1c_2\cdots c_{k-1} \to \U^{c_1}\D \U^{c_2-c_1} \D\cdots \U^{c_{k-1}-c_{k-2}} \D. \]

On the other hand, observe that $\Gr_{2k-3}(12\cdots k) = \A_{k-2}\,\dot\cup\, \A_{k-1}$, where $\A_{k-2}$ is the subset of permutations in $\Gr_{2k-3}(12\cdots k)$ having their descent at position $k-2$, and $\A_{k-1}$ is the subset of elements having their descent at position $k-1$. Removing the largest entry from $\pi\in \Gr_{2k-2}(12\cdots k)$ gives a unique element of $\A_{k-2}$, and removing 1 from $\pi$ gives a unique element of $\A_{k-1}$. In other words, $\abs{\A_{k-2}} = \abs{\Gr_{2k-2}(12\cdots k)}=\abs{\A_{k-1}}$, and therefore we have $\abs{\Gr_{2k-3}(12\cdots k)} = 2C_{k-1}$.
\end{proof}

More generally, Michael Weiner suggested the following formula, which we were able to verify up to $k=12$, see Table~\ref{tab:finiteClasses}.

\begin{conjecture}[Weiner]
For $k\ge 2$ and $m\in\{k,\dots,2k-2\}$, 
\[ \abs{\Gr_{m}(12\cdots k)} = \sum_{j=1}^{k-\lfloor m/2\rfloor} (-1)^{j-1} j\cdot \binom{2k-m-j}{j}C_{k-j}. \]
\end{conjecture}

\begin{center}
\begin{table}[ht]
\small
\begin{tabular}{r|l}
\R $k$ & \hspace{.2\textwidth} $\abs{\Gr_{k}(12\cdots k)},\; \dots\;, \abs{\Gr_{2k-2}(12\cdots k)}$ \\ \hline
\R 2  & 1 \\
3  & 4, 2 \\
4  & 11, 10, 5 \\
5  & 26, 32, 28,14 \\
6  & 57, 84, 98, 84, 42 \\
7  & 120, 198, 276, 312, 264, 132 \\
8  & 247, 438, 687, 924, 1023, 858, 429 \\
9  & 502, 932, 1584, 2398, 3146, 3432, 2860,1430 \\
10 & 1013, 1936, 3476, 5720, 8437, 10868, 11726, 9724, 4862 \\
11 & 2036, 3962, 7384, 12896, 20696, 29926, 38012, 40664, 33592, 16796 \\
12 & 4083, 8034, 15353, 27976, 47762, 75140, 106964, 134368, 142766, 117572, 58786
\end{tabular}
\bigskip
\caption{Number of permutations in $\Gr_{m}(12\cdots k)$ for $k\le m\le 2k-2$.}
\label{tab:finiteClasses}
\end{table}
\end{center}

\medskip
For the rest of this section, we will focus on the avoidance of patterns having one descent.
\begin{lem}\label{lem:size3pattern}
For $\sigma\in \{132, 213, 231, 312\}$ and $n\in\mathbb{N}$, we have
\[ \abs{\Gr_n(\sigma)} = 1 + \binom{n}{2}. \]
\end{lem}
\begin{proof}
Since $\abs{\Gr_n(\sigma)}=\abs{\Gr_n(\sigma^{rc})}$, it is enough to only consider the patterns $132$ and $231$. Clearly, $\pi=1\cdots n\in\Gr_n(\sigma)$ if $\des(\sigma)>0$. Moreover, the permutation $\pi=n\,1\cdots(n-1)$ as well as any permutation of the form 
\[ \pi = i\,1\cdots(i-1)(i+1)\cdots n \;\text{ for } i\in\{2,\dots,n-1\}, \]
all avoid the patterns 132 and 231. These are the $n$ permutations in $\Gr_n(132)\cap \Gr_n(231)$. Note that if $\pi$ has a descent at position $k>1$, then $\pi$ must contain either a 132 pattern (if $\pi(k-1)<\pi(k+1)$) or a 231 pattern (if $\pi(k-1)>\pi(k+1)$). Thus $\pi\not\in \Gr_n(132)\cap \Gr_n(231)$.

If $\pi\in \Gr_n(132)\setminus \Gr_n(231)$, then it must be of the form 
\[ \pi=i(i+1)\cdots j\,1\,\tau \;\text{ with } i,j\in\{2,\dots n\}, \; i<j, \] 
where $\tau$ is the word consisting of the remaining $n-(j-i+2)$ elements of $[n]$ written in increasing order. There are $\binom{n-1}{2}$ such permutations, so
\[ \abs{\Gr_n(132)} = n+\binom{n-1}{2} = 1 + \binom{n}{2}. \]
Similarly, if $\pi\in \Gr_n(231)\setminus \Gr_n(132)$, then it must be of the form 
\[ \pi=1\cdots (i-1)\, j\, i\,\tau \;\text{ with } i,j\in\{2,\dots n\}, \; i<j, \]
leading to the same number of permutations. 
\end{proof}

Note that $\Gr_n(132) = S_n(132,321)$, so we could have proved the above lemma by means of \cite[Prop.~11]{SiSch85} together with an explicit bijection from $\Gr_n(132)$ to $\Gr_n(231)$.

More generally, for patterns of size $k\ge 3$ having only one descent, there is only one Wilf equivalence class of pattern-avoiding Grassmannian permutations: 
\begin{thm}\label{thm:mainResult}
If $k\ge 3$ and $\sigma\in S_k$ with $\des(\sigma)=1$, then
\[ \abs{\Gr_n(\sigma)}  = 1 + \sum_{j=3}^k\binom{n}{j-1} \text{ for } n\in\mathbb{N}. \]
\end{thm}
\begin{proof}
For $k = 3$, the statement was proven in Lemma~\ref{lem:size3pattern}. We proceed by induction in $k$. 

We start by discussing how the statement for $k=4$ follows from the case when $k=3$. Note that, for any $\sigma \in S_{4}$ with $\des(\sigma)=1$, there is a $\sigma' \in S_{3}$ with $\des(\sigma)=1$ obtained from $\sigma$ by removing either the 1 or the 4 as shown in the following table. 

\smallskip
\begin{center}
\small
\begin{tabular}{c||l}
\R  $\sigma'$ & \hspace{5ex} Pattern $\sigma$ \\ \hline
\R 132 & \textbf{1}243, 132\textbf{4}, 13\textbf{4}2, 24\textbf{1}3 \\ \hline
\R 213 & \textbf{1}324, 213\textbf{4}, 2\textbf{4}13, 3\textbf{1}24 \\ \hline
\R 231 & \textbf{1}342, 231\textbf{4}, 23\textbf{4}1, 34\textbf{1}2 \\ \hline
\R 312 & \textbf{1}423, 312\textbf{4}, 3\textbf{4}12, 4\textbf{1}23
\end{tabular}
\end{center}
In such instances where there are two options for $\sigma'$, either choice leads to the same result.

For example, suppose $\sigma = 2413$ and choose $\sigma'=213$. Since $\Gr_n(213)\subset \Gr_n(2413)$, the set $\Gr_n(2413)$ can be decomposed into disjoint sets
\[  \Gr_n(2413) = \Gr_n(213) \: \dot\cup \: \big( \Gr_n(2413)\setminus \Gr_n(213)\big). \]
We already know how to count $\Gr_n(213)$, so it remains to enumerate $\Gr_n(2413)\setminus \Gr_n(213)$. This is the set of permutations in $\Gr_n(2413)$ that contain a $213$ pattern, so their graph must be of the form

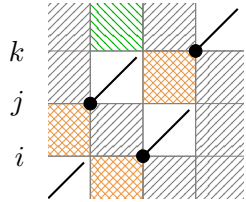
\begin{center}
\begin{tikzpicture}[scale=0.7]
\node[left=1] at (0, 1) {$i$};
\node[left=1] at (0, 2) {$j$};
\node[left=1] at (0, 3) {$k$};
\clip (0.2,0.2) rectangle (3.9,3.9);
\draw[gray] (0,0) grid (4,4);
\draw[mesh] (1,1) rectangle (2,2); 
\draw[mesh] (0,2) rectangle (1,4); 
\draw[mesh] (3,0) rectangle (4,3);
\draw[mesh] (2,0) rectangle (3,1);
\draw[mesh] (2,3) rectangle (3,4);
\draw[meshPatt] (1,3) rectangle (2,4);
\foreach \x/\y in {0/1,1/0,2/2}{\draw[meshMost] (\x,\y) rectangle (1+\x,1+\y);}
\draw[thick] (0.2,0.2) -- (0.88,0.88);
\draw[thick] (1,2) -- (1.88,2.88);
\draw[thick] (2,1) -- (2.88,1.88);
\draw[thick] (3,3) -- (3.8,3.8);
\foreach \x/\y in {1/2,2/1,3/3}{\draw[fill] (\x,\y) circle (0.12);}
\end{tikzpicture}
\end{center}
where the dots and regions shaded with \tikz[scale=0.7,baseline=2pt]{\draw[meshMost] (0,0) rectangle (0.5,0.5);} correspond to the choice of a left-most $213$ pattern, and the other shaded regions represent the restrictions of being Grassmannian (\tikz[scale=0.7,baseline=2pt]{\draw[mesh] (0,0) rectangle (0.5,0.5);}) and $2413$-avoiding (\tikz[scale=0.7,baseline=2pt]{\draw[meshPatt] (0,0) rectangle (0.5,0.5);}). More precisely, every $\pi\in \Gr_n(2413)\setminus \Gr_n(213)$ is of the form 
\[ \pi = \tau_0\, j\, \tau_1\, i\, \tau_2\, k\, \tau_3 \;\text{ with } 1\le i<j<k \le n, \]
where the permutations $\tau_0$, $\tau_1$, $\tau_2$, $\tau_3$ are either empty or consist of increasing consecutive numbers. In other words, $\pi$ is uniquely determined by a choice of $i, j, k\in[n]$. Therefore,
\[ \abs{\Gr_n(2413)\setminus \Gr_n(213)}=\binom{n}{3}\; \text{ and }\; \abs{\Gr_n(2413)}=1+\binom{n}{2}+\binom{n}{3}. \]
This coincides with the result from Proposition~\ref{prop:biGrass}.

For any other pattern $\sigma$ of size $4$ with $\des(\sigma)=1$, the argument is similar. The only difference occurs depending on whether the choice of $\sigma'$ is the result of removing the $1$ or the $4$ from $\sigma$. If the $1$ is removed, we take the right-most $\sigma'$-pattern for the plot, and if the $4$ is removed, then we take the left-most $\sigma'$-pattern instead. In any case, the elements of $\Gr_n(\sigma)\setminus \Gr_n(\sigma')$ are determined by the choice of a left-most/right-most $\sigma'$-pattern. Table~\ref{tab:patternSamples} shows examples of generic plots for other corresponding $(\sigma,\sigma')$ pairs. In conclusion, for $\sigma\in S_{4}$ with $\des(\sigma)=1$, we have $\abs{\Gr_n(\sigma)}  = 1 + \binom{n}{2} + \binom{n}{3}$.

\begin{table}[ht]
\small
\begin{tabular}{c|c|c|c}
\R $\Gr_n(2134)\backslash \Gr_n(213)$ & $\Gr_n(2341)\backslash \Gr_n(231)$ & $\Gr_n(2413)\backslash \Gr_n(132)$ & $\Gr_n(3412)\backslash \Gr_n(312)$ \\[2pt] \hline
\rule{0pt}{70pt}\noindent
\begin{tikzpicture}[scale=0.6]
\clip (0.1,0.1) rectangle (3.9,3.9);
\draw[gray] (0,0) grid (4,4);
\draw[mesh] (0,2) rectangle (1,4); 
\draw[mesh] (1,1) rectangle (2,2); 
\draw[mesh] (2,3) rectangle (3,4);
\draw[mesh] (2,0) rectangle (3,1);
\draw[mesh] (3,0) rectangle (4,3);
\draw[meshPatt] (3,3) rectangle (4,4);
\foreach \x/\y in {0/1,1/0,2/2}{\draw[meshMost] (\x,\y) rectangle (1+\x,1+\y);}
\draw[thick] (0.2,0.2) -- (0.88,0.88);
\draw[thick] (2,1) -- (2.88,1.88);
\draw[thick] (1,2) -- (1.88,3.88);
\foreach \x/\y in {1/2,2/1,3/3}{\draw[fill] (\x,\y) circle (0.12);}
\draw[fill=white] (1.47,3) circle (0.12);
\end{tikzpicture}
&
\begin{tikzpicture}[scale=0.6]
\clip (0.1,0.1) rectangle (3.9,3.9);
\draw[gray] (0,0) grid (4,4);
\draw[mesh] (0,2) rectangle (1,4); 
\draw[mesh] (1,0) rectangle (2,2); 
\draw[mesh] (1,3) rectangle (2,4);
\draw[mesh] (2,1) rectangle (3,3);
\draw[mesh] (3,0) rectangle (4,1);
\draw[meshPatt] (2,3) rectangle (3,4);
\foreach \x/\y in {0/1,1/2,2/0}{\draw[meshMost] (\x,\y) rectangle (1+\x,1+\y);}
\draw[thick] (0.2,0.2) -- (0.88,0.88);
\draw[thick] (3,1) -- (3.8,3.8);
\foreach \x/\y in {1/2,2/3,3/1}{\draw[fill] (\x,\y) circle (0.12);}
\draw[fill=white] (3.3,2) circle (0.12);
\draw[fill=white] (3.57,3) circle (0.12);
\end{tikzpicture}
&
\begin{tikzpicture}[scale=0.6]
\clip (0.1,0.1) rectangle (3.9,3.9);
\draw[gray] (0,0) grid (4,4);
\draw[mesh] (1,0) rectangle (2,1); 
\draw[mesh] (0,1) rectangle (1,4); 
\draw[mesh] (1,3) rectangle (2,4);
\draw[mesh] (2,2) rectangle (3,3);
\draw[mesh] (3,0) rectangle (4,3);
\draw[meshPatt] (2,0) rectangle (3,1);
\foreach \x/\y in {1/1,2/3,3/2}{\draw[meshMost] (\x,\y) rectangle (1+\x,1+\y);}
\draw[thick] (0.2,0.2) -- (1,1);
\draw[thick] (1.12,2.12) -- (2,3);
\draw[thick] (2.12,1.12) -- (3,2);
\draw[thick] (3.12,3.12) -- (3.8,3.8);
\foreach \x/\y in {1/1,2/3,3/2}{\draw[fill] (\x,\y) circle (0.12);}
\end{tikzpicture}
&
\begin{tikzpicture}[scale=0.6]
\clip (0.1,0.1) rectangle (3.9,3.9);
\draw[gray] (0,0) grid (4,4);
\draw[mesh] (0,3) rectangle (1,4); 
\draw[mesh] (1,1) rectangle (2,3); 
\draw[mesh] (2,0) rectangle (3,1);
\draw[mesh] (2,2) rectangle (3,4);
\draw[mesh] (3,0) rectangle (4,2);
\draw[meshPatt] (1,3) rectangle (2,4);
\foreach \x/\y in {0/2,1/0,2/1}{\draw[meshMost] (\x,\y) rectangle (1+\x,1+\y);}
\draw[thick] (0.2,0.2) -- (0.9,1.88);
\draw[thick] (3,2) -- (3.8,3.8);
\foreach \x/\y in {1/3,2/1,3/2}{\draw[fill] (\x,\y) circle (0.12);}
\draw[fill=white] (0.53,1) circle (0.12);
\draw[fill=white] (3.45,3) circle (0.12);
\end{tikzpicture}
\end{tabular}
\bigskip
\caption{Examples for $\Gr_n(\sigma)\setminus \Gr_n(\sigma')$.}
\label{tab:patternSamples}
\end{table}

Suppose now that the statement of the theorem is true for patterns of size $k-1$ with only one descent, and let $\sigma\in S_k$ with $\des(\sigma)=1$. Then $\sigma$ must be of the form $\sigma=1\oplus\sigma'$ or $\sigma=\sigma'\oplus 1$ with $\des(\sigma')=1$, or $\sigma=\sigma_L k1 \sigma_R$ with possibly empty or all increasing permutations $\sigma_L$ and $\sigma_R$. In all cases, by removing either the $1$ (if $\sigma=1\oplus\sigma'$ or if $\sigma_L=\varepsilon$) or the $k$ (if $\sigma=\sigma'\oplus 1$ or $\sigma_R=\varepsilon$), we arrive at a permutation $\sigma'$ with $\des(\sigma')=1$.

Choose such a $\sigma'$ and split $\Gr_n(\sigma) = \Gr_n(\sigma') \: \dot\cup \: \big(\Gr_n(\sigma)\setminus \Gr_n(\sigma')\big)$. By the induction hypothesis, $\abs{\Gr_n(\sigma')} = 1 + \sum_{j=3}^{k-1}\binom{n}{j-1}$. Now, in order to count the elements of $\Gr_n(\sigma)\setminus \Gr_n(\sigma')$, we consider a $k \times k$ grid and plot the $k-1$ elements of the contained pattern $\sigma'$. This divides the grid into $k^2$ regions with some of the regions shaded based on the restrictions on the permutations -- being Grassmannian (having only one descent), avoidance of the pattern $\sigma$, and the choice of $\sigma'$ being the left-most or right-most pattern. 

Through this process, we will end up with exactly $k$ unique unshaded regions, each one in a different row of the grid, determining the elements of a permutation in $\Gr_n(\sigma)\setminus \Gr_n(\sigma')$. 

First, from the $k^2$ available regions, $k(k-2)$ of them will be shaded because of the Grassmannian condition, leaving two unshaded regions in each row. Moreover, if $\sigma'$ is obtained from $\sigma$ by removing $k$ (like in our example for $k=4$), then we declare $\sigma'$ to be the left-most instance of the pattern and shade the regions having the elements of $\sigma'$ on their top right corner. We also shade the region on the top row of the grid that could produce a $\sigma$-pattern. Similarly, if $\sigma'$ is obtained from $\sigma$ by removing the $1$, then we pick the right-most instance of $\sigma'$ and shade the regions having the elements of $\sigma'$ on their bottom left corner. In addition, we shade the region on the bottom row that could produce a $\sigma$-pattern.

In both cases, the choice of $\sigma'$ induces the shading of $k-1$ regions of the grid, and the $\sigma$-avoiding condition adds one more forbidden region to the graph. In total, there will be $k^2 - k(k-2) - (k-1) - 1 = k$ distinguishable unshaded regions available to create a permutation in $\Gr_n(\sigma)\setminus \Gr_n(\sigma')$ with increasing runs of consecutive numbers.

In conclusion, every permutation in $\Gr_n(\sigma)\setminus \Gr_n(\sigma')$ is uniquely determined by the choice of the $k-1$ elements of the left-most/right-most instance of $\sigma'$, so $\abs{\Gr_n(\sigma)\setminus \Gr_n(\sigma')} = \binom{n}{k-1}$ and therefore $\abs{\Gr_n(\sigma)} = 1 + \sum_{j=3}^{k}\binom{n}{j-1}$.
\end{proof}

Some of the sequences generated by the enumeration of $\Gr_n(\sigma)$ are listed in Table~\ref{tab:avoidingGrass}.
\begin{table}[ht]
\small
\begin{tabular}{|c|l|c|} \hline
\R $\;\abs{\sigma}\;$ & \hspace{35pt} Sequence $\abs{\Gr_n(\sigma)}$ & OEIS \\[2pt] \hline
\R  3& $1, 2, 4, 7, 11, 16, 22, 29, 37, 46,\dots$ &  A000124 \\ \hline
\R  4& $1, 2, 5, 11, 21, 36, 57, 85, 121, 166,\dots$ & A050407 \\ \hline
\R  5& $1, 2, 5, 12, 26, 51, 92, 155, 247, 376,\dots$ & A027927 \\ \hline
\R  6& $1, 2, 5, 12, 27, 57, 113, 211, 373, 628,\dots$ & n/a\\ \hline
\R  7& $1, 2, 5, 12, 27, 58, 120, 239, 457, 838,\dots$ & n/a\\ \hline
\R  8& $1, 2, 5, 12, 27, 58, 121, 247, 493, 958,\dots$ & n/a\\ \hline
\R  9& $1, 2, 5, 12, 27, 58, 121, 248, 502, 1003,\dots$ & n/a\\ \hline
\R  10& $1, 2, 5, 12, 27, 58, 121, 248, 503, 1013,\dots$ & n/a\\ \hline
\end{tabular}
\bigskip
\caption{Enumeration of $\Gr_n(\sigma)$ for a pattern $\sigma$ with $\des(\sigma)=1$.}
\label{tab:avoidingGrass}
\end{table}

\section{Connection to Dyck and Schr\"oder paths}
\label{sec:bijections}

The OEIS gives several combinatorial interpretations for some of the sequences mentioned in this paper. In this section, we give explicit bijections connecting (pattern-avoiding) Grassmannian permutations to Dyck paths and certain pattern-avoiding Schr\"oder paths.

Let $\Dyck(n)$ denote the set of Dyck paths of semilength $n$. We consider a bijective map 
\begin{equation*}
 \varphi: \Dyck(n) \to S_n(321)
\end{equation*}
that is particularly amenable to Grassmannian permutations. It is defined as follows:
\begin{itemize}
\item[-] From left to right, number the down-steps of the Dyck path with $[n]$ in increasing order. 
\item[-] At each peak $\U\D$, label the up-step with the number already assigned to its paired down-step.
\item[-] Going through the ascents from left to right, label the remaining up-steps from bottom to top on each ascent in a greedy fashion.
\item[-] The resulting labeling gives a $321$-avoiding permutation on $[n]$.
\end{itemize}
For example, the path in Figure~\ref{fig:peak-greedy} gives the permutation $23174586\in S_8(321)$.

\medskip
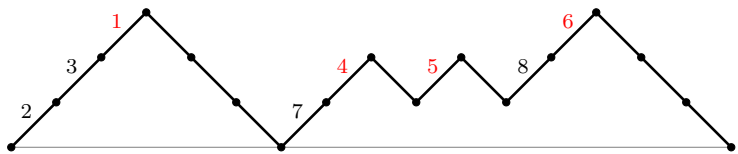
\begin{figure}[ht] 
\begin{tikzpicture}[scale=0.6]
\scriptsize
	\dyckpath{0,0}{8}{1,1,1,0,0,0,1,1,0,1,0,1,1,0,0,0};
	\draw[red] (2.7,2.8) node[left=1pt] {1};
	\draw[red] (7.7,1.8) node[left=1pt] {4};
	\draw[red] (9.7,1.8) node[left=1pt] {5};
	\draw[red] (12.7,2.8) node[left=1pt] {6};
	\draw (1.7,1.8) node[left=1pt] {3};
	\draw (0.7,0.8) node[left=1pt] {2};
	\draw (6.7,0.8) node[left=1pt] {7};
	\draw (11.7,1.8) node[left=1pt] {8};
\end{tikzpicture}
\caption{$P\in\Dyck(8)$ with $\varphi(P)=23174586$.}
\label{fig:peak-greedy}
\end{figure}

Observe that labels at the peaks of a Dyck path $P$ correspond to the right-to-left (RL) minima of the permutation $\varphi(P)$. We have $\varphi\big((\U\D)^n\big) = 12\cdots n$, and the number of descents of $\varphi(P)$ is equal to the number of long ascents (2 or more up-steps) in $P$.

The inverse of $\varphi$ is defined as follows. Given a permutation $\pi\in S_n(321)$ with RL minima $\pi(i_1)<\pi(i_2)<\cdots <\pi(i_\ell)$, write it as $\pi = \tau_{i_1}\cdots \tau_{i_\ell}$, where each $\tau_{i_k}$ ends with the RL minimum $\pi(i_k)$. If $a_k = \abs{\tau_{i_k}}$ and $b_k = \pi(i_{k+1})-\pi(i_k)$ for $k\in\{1,\dots,\ell-1\}$, then we let
\[ \varphi^{-1}(\pi) = \U^{a_1} \D^{b_1}\U^{a_2} \D^{b_2}\cdots \U^{a_{\ell-1}} \D^{b_{\ell-1}} \U^{\abs{\tau_{i_\ell}}} \D^{n+1-\pi(i_\ell)}. \]
For example, if $\pi = 23{\bf 1}7{\bf 4}{\bf 5}8{\bf 6}$, then $\tau_1=231$, $\tau_2=74$, $\tau_3=5$, and $\tau_4=86$, so the above construction gives $\varphi^{-1}(\pi)=\U^3\D^3\U^2\D^1\U^1\D^1\U^2\D^3$. This is the Dyck path in Figure~\ref{fig:peak-greedy}.

As an immediate consequence of the definition of $\varphi$, we get the following correspondence.
\begin{prop}
The set $\Gr_n$ of Grassmannian permutations on $[n]$ is in bijection with the set of Dyck paths of semilength $n$ having at most one long ascent.  
\end{prop}

We will call the elements of $\varphi^{-1}(\Gr_n)$ {\em Grassmannian Dyck paths}. The path $(\U\D)^n$ is the only path of height 1, and we will refer to it as the {\em identity path}. The height of any other Grassmannian Dyck path is the length of its long ascent. All peaks before the long ascent happen at height 1, and the heights of peaks after the long ascent form a weakly decreasing sequence. More precisely, every Grassmannian Dyck path different from the identity must be of the form
\begin{equation}\label{eq:GrassDyck}
 (\U\D)^{\ell_1} \U P_{k}\D (\U\D)^{\ell_2} \;\text{ with } \ell_1, \ell_2\ge 0,\; \ell_1+\ell_2<n-1, 
\end{equation}
where $P_{k}$ is a Dyck path of semilength $n-1-\ell_1-\ell_2$ having $k$ peaks and avoiding $\U\U$ except possibly on its first ascent. In particular, $\U P_k\D$ is an indecomposable Grassmannian Dyck path with $k$ peaks. Moreover,
\begin{equation}\label{eq:GrassReduction}
 \varphi\big((\U\D)^{\ell_1} \U P_{k}\D (\U\D)^{\ell_2}\big) = \id_{\ell_1}\oplus\; \pi_k \oplus \id_{\ell_2}
\end{equation}
with $\pi_k\in\Gr_{n-\ell_1-\ell_2}$ having $k$ RL minima. 

\begin{prop}
For $k\ge 3$, the elements of $\Gr_n(k12\cdots (k-1))$ are in one-to-one correspondence with the Grassmannian Dyck paths of semilength $n$ having at most $k-2$ peaks at height greater than 1.
\end{prop}
\begin{proof}
The map $\varphi$ provides the claimed bijection. Observe that because of the discussion around \eqref{eq:GrassDyck} and \eqref{eq:GrassReduction}, it suffices to consider Grassmannian Dyck paths that start with a long ascent and have no peaks at height 1. If $P$ is such a path of semilength $n$ with $\ell$ peaks, then $\varphi(P)$ is a Grassmannian permutation of the form
\[ \varphi(P) =  \tau_0\, n\, 1\,j_2\cdots j_\ell \;\text{ with } j_2<\cdots< j_\ell \text{ and } \abs{\tau_0} = n-\ell-1. \]
Therefore, $\varphi(P)\in \Gr_n(k12\cdots (k-1))$ if and only if $\ell<k-1$.

Suppose now that $\pi\in \Gr_n(k12\cdots (k-1))$ is indecomposable and has $m$ RL minima. Then, the entries $n$ and $1$ make the descent of $\pi$ and we must have $m<k-1$. Hence $\varphi^{-1}(\pi)$ is a Dyck path with $m\le k-2$ peaks at height greater than 1. 
\end{proof}

\begin{prop}
For $k\ge 3$, the elements of $\Gr_n(23\cdots k1)$ are in one-to-one correspondence with the Grassmannian Dyck paths of semilength $n$ and height at most $k-1$.
\end{prop}
\begin{proof}
Once again, because of \eqref{eq:GrassReduction}, we only need to consider the correspondence between indecomposable permutations and Dyck paths that start with a long ascent and have no peaks at height 1. Such a path of height $h$ must be of the form
\[ P = \U^h\D^{b_0}\U\D^{b_1}\cdots \U\D^{b_\ell} \;\text{ with } h\ge 2 \text{ and } b_i\ge 1, \]
which implies $\varphi(P) = \tau_{h-1}\, 1\,\tau_\ell$ with increasing permutations $\tau_{h-1}$, $\tau_\ell$, and
$\abs{\tau_{h-1}} = h-1$. Clearly, $\varphi(P)\in \Gr_n(23\cdots k1)$ if and only if $h\le k-1$. Conversely, every indecomposable permutation $\pi\in\Gr_n(23\cdots k1)$ must start with an increasing run of size $\ell<k-1$, followed by the entry 1. So, the height of the corresponding Dyck path $\varphi^{-1}(\pi)$ is $\ell+1$ and has therefore height at most $k-1$.
\end{proof}

\bigskip
We finish this section with an interesting connection between the elements of $\Gr_n(35124)$ and certain Schr\"oder paths.

\begin{lem}
We have $\pi\in\Gr_{n+1}(35124)$ if and only if its Lehmer code is of the form 
\begin{equation}\label{eq:35124Lehmer}
 L(\pi) = 0^{j_1} 1^{j_2} m^{j_3} 0^{j_4} =     
 \underbrace{0\cdots 0}_{j_{1}}\; \underbrace{1\cdots 1}_{j_{2}}\; 
 \underbrace{m\cdots m}_{j_{3}}\; \underbrace{0\cdots 0}_{j_{4}}\,,
\end{equation}
where $j_1+\cdots+j_4=n+1$, $j_4>0$, $m\in\{2,\dots,n\}$, and $m\le j_4$.
\end{lem}
\begin{proof}
Clearly, $L(\id_{n+1}) = 0^{n+1}$. Suppose $\pi\in\Gr_{n+1}(35124)$ with $\des(\pi)=1$. If $\pi$ has its descent at position 1, then $\pi = (\ell+1)\,1\cdots \ell \cdots$ for some $1\le \ell\le n$, and thus $L(\pi) = \ell\, 0^{n}$. Otherwise, the shape of $\pi$ depends on the avoidance or containment of a $2413$ pattern. More precisely, if $\pi$ does not have its descent at position $1$, then it must be of the form

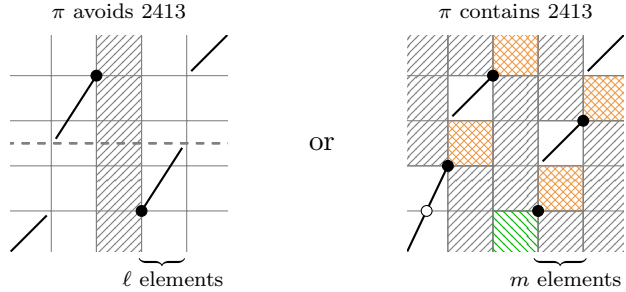
\begin{center}
\begin{tikzpicture}[scale=0.6]
\begin{scope}
\node[above=1pt] at (2.5,5) {\scriptsize $\pi$ avoids $2413$};
\clip (0.1,0.1) rectangle (4.9,4.9);
\draw[gray] (0,0) grid (5,5);
\draw[mesh] (2,0) rectangle (3,5); 
\draw[dashed,gray,thick] (0,2.5) rectangle (5,2.5); 
\foreach \x/\y in {2/4,3/1}{\draw[fill] (\x,\y) circle (0.12);}
\draw[thick] (0.1,0.1) -- (0.9,0.9);
\draw[thick] (1.1,2.6) -- (2,4);
\draw[thick] (3,1) -- (3.9,2.4);
\draw[thick] (4.1,4.1) -- (4.9,4.9);
\end{scope}
\draw[decorate, decoration = {calligraphic brace, mirror, raise=1pt}, thick] (2.95,0) -- (3.95,0);
\node[below=2pt] at (3.7,0) {\scriptsize $\ell$ elements};

\node at (7,2.5) {or};

\begin{scope}[xshift=250]
\node[above=1pt] at (2.5,5) {\scriptsize $\pi$ contains $2413$};
\clip (0.1,0.1) rectangle (4.9,4.9);
\draw[gray] (0,0) grid (5,5);
\draw[mesh] (0,2) rectangle (1,5); 
\draw[mesh] (1,0) rectangle (2,2); 
\draw[mesh] (2,1) rectangle (3,4); 
\draw[mesh] (3,3) rectangle (4,5); 
\draw[mesh] (4,0) rectangle (5,3);
\draw[meshPatt] (2,0) rectangle (3,1);
\foreach \x/\y in {1/4,3/0}{\draw[mesh] (\x,\y) rectangle (1+\x,1+\y);}
\foreach \x/\y in {1/2,2/4,3/1,4/3}{\draw[meshMost] (\x,\y) rectangle (1+\x,1+\y);}
\foreach \x/\y in {1/2,2/4,3/1,4/3}{\draw[fill] (\x,\y) circle (0.12);}
\draw[thick] (0.1,0.1) -- (1,2);
\draw[thick] (1.1,3.1) -- (2,4);
\draw[thick] (3.1,2.1) -- (4,3);
\draw[thick] (4.1,4.1) -- (4.9,4.9);
\draw[fill=white] (0.53,1) circle (0.12);
\end{scope}
\draw[decorate, decoration = {calligraphic brace, mirror, raise=1pt},thick] (11.7,0) -- (12.85,0);
\node[below=2pt] at (12.4,0) {\scriptsize $m$ elements};
\end{tikzpicture}
\end{center}
leading to the codes $0^{k_1} \ell^{k_2} 0^{k_3}$ with $1\le \ell\le k_3$ or $0^{k_1} 1^{k_2} m^{k_3} 0^{k_4}$ with $2\le m\le k_4$. Combining all the above cases we get that, if $\pi\in\Gr_{n+1}(35124)$, then $L(\pi)$ is of the form \eqref{eq:35124Lehmer}.

On the other hand, if $\pi$ contains a $35124$ pattern, then it must be of the form
\[ \pi = \tau_1\, c\, \tau_2\, e\, \tau_3\, a\, \tau_4\, b\, \tau_5\, d\, \tau_6 \text{ with } a<b<c<d<e, \]
where every $\tau_i$ is either empty or $21$-avoiding, and the elements of $\tau_2$ (if any) are all larger than $c$. Let $i_a, i_b, i_c, i_d ,i_e$ denote the positions of the entries $a, b, c, d, e$, respectively. Note that $(i_c,i_a)$ and $(i_c,i_b)$ are inversions, so the Lehmer code at position $i_c$ is some $\ell_c\ge 2$. Moreover, $(i_e,i_d)$ is an inversion, and if $(i_c,j)$ is an inversion, so is $(i_e,j)$. Therefore, the code at position $i_e$ is some $\ell_e>\ell_c$. In conclusion, the containment of a $35124$ pattern guarantees a Lehmer code $L(\pi)$ with at least two different letters greater than $1$, hence it is not of the form \eqref{eq:35124Lehmer}.
\end{proof}

\medskip
A Schr\"oder path of semilength $n$ is a lattice from $(0,0)$ to $(2n,0)$ using steps $\U=(1,1)$, $\D=(1,-1)$, and $\H=(2,0)$, never going below the $x$-axis. To each of these paths we can associate a word over the alphabet $\{\U,\D,\H\}$ with $\val(\U)=1$, $\val(\D)=-1$, and $\val(\H)=0$. A {\em Schr\"oder word} is a word $w$ over that alphabet such that $\val(w) = 0$, and if $w=uv$, then $\val(u) \geq 0$. Note that if $w$ corresponds to a Schr\"oder path of semilength $n$, then the number of letters in $w$ must satisfy $\#\U + \#\D + 2(\#\H) = 2n$.

We let $\Schr_{n}(\U\U\D\D)$ be the set of Schr\"oder words $w$ such that
\begin{enumerate}[\qquad (i)]
\item $w = uv \implies \val(u) \in \{0,1\}$,
\item $(\#\U\text{ in } w) \le 2$.
\end{enumerate}

Conditions (i) and (ii) imply that every element of $\Schr_{n}(\U\U\D\D)$ must be of the form
\[ w=\H^n \;\text{ or }\; w = w_{1}\U w_{2}\D w_{3} \;\text{ or }\; w = w_{1}\U w_{2}\D w_{3}\U w_{4}\D w_{5}, \]
where each $w_{i}$ is either empty or consists of $\H$ steps only. In other words, $\Schr_{n}(\U\U\D\D)$ is the set of Schr\"oder words that avoid the pattern, $\U\U\D\D$, not necessarily consecutively. For more on pattern avoiding Schr\"oder words, see \cite{CioFer17}.

For $w=u_1\cdots u_\ell \in \Schr_{n}(\U\U\D\D)$ with $u_i\in\{\U,\D,\H\}$, we let $\bin(w)$ be the binary word defined by
\[ \bin(w)_j = \val(u_1\cdots u_j)  \text{ for every } j\in\{1,\dots,\ell\}. \]
For example,
\begin{equation*}
\bin(\H\H\H\H\H) = 00000, \hspace{6pt}
\bin(\H\U\H\H\D\H) = 011100, \hspace{6pt}
\bin(\U\H\D\H\U\D\H) = 1100100.
\end{equation*}

In general, for $w\in \Schr_{n}(\U\U\D\D)$, we have
\begin{equation*}
 \abs{\bin(w)} = n + \#\U,
\end{equation*}
where $\#\U$ is $0$, $1$ or $2$. In particular, $w=\H^n$ is the only word of length $n$ in $\Schr_{n}(\U\U\D\D)$. As stated above, if $w$ has only one $\U$, then
\[ w = w_{1}\U w_{2}\D w_{3}  \;\text{ and therefore } \bin(w) = 0^{i_1} 1^{i_2} 0^{i_3}, \]
where $i_{1} = \abs{w_{1}}$, $i_{2} = \abs{\U w_{2}}$, $i_{3} = \abs{\D w_{3}}$, and $i_{1} + i_{2} + i_{3} = n+1$.

Finally, if $w$ has two $\U$'s, then
\[    w = w_{1}\U w_{2}\D w_{3}\U w_{4}\D w_{5} \;\text{ and so } \bin(w) = 0^{i_1} 1^{i_2} 0^{i_3} 1^{i_4} 0^{i_5}, \]
where $i_{1} = \abs{w_{1}}$, $i_{2} = \abs{\U w_{2}}$, $i_{3} = \abs{\D w_{3}}$, $i_{4} = \abs{\U w_{4}}$, $i_{5} = \abs{\D w_{5}}$, and $i_{1} + \dots + i_{5} = n+2$.

\begin{prop}
The set $\Schr_n(\U\U\D\D)$ is in bijection with $\Gr_{n+1}(35124)$.
\end{prop}
\begin{proof}
For $w \in \Schr_{n}(\U\U\D\D)$, we define $\alpha$ by
\begin{equation*}
  \alpha(w) = 
  \begin{cases}
  0^{n+1} & \text{ if $w$ has no } \U, \\
  \bin(w) & \text{ if $w$ has only one } \U, \\
  0^{i_1}1^{i_2-1}(i_3+1)^{i_4}0^{i_3+i_5}& \text{ if } \bin(w) = 0^{i_1} 1^{i_2} 0^{i_3} 1^{i_4} 0^{i_5}. 
  \end{cases}
\end{equation*}
For example,
\begin{align*}
\alpha(\H\H\H\H\H) &= 000000, \\
\alpha(\H\U\H\H\D\H) &= 011100, \text{ and} \\
\alpha(\U\H\D\H\U\D\H) &= 130000 \,\text{ since } \bin(w) = 1100100.  
\end{align*}
By definition, $\alpha(w)$ is a word of length $n+1$ of the form \eqref{eq:35124Lehmer}, hence it is the Lehmer code of a permutation in $\Gr_{n+1}(35124)$. In fact, the map $L^{-1}\circ \alpha: \Schr_n(\U\U\D\D)\to \Gr_{n+1}(35124)$ is a bijection whose inverse is defined as follows.
 
If $\pi\in \Gr_{n+1}(35124)$, then $L(\pi) = 0^{j_1} 1^{j_2} m^{j_3} 0^{j_4}$ with $j_1+\cdots+j_4=n+1$, $m\in\{2,\dots,n\}$, $j_4>0$, and $m\le j_4$. If $j_2=j_3=0$, then $L(\pi) = 0^{n+1}$ and so $\alpha^{-1}(L(\pi)) = \H^{n}$. If $j_3=0$ and $j_2>0$, then $L(\pi) = 0^{j_1} 1^{j_2} 0^{j_4}$ and we have $\alpha^{-1}(L(\pi)) =\H^{j_1}\U\H^{j_2-1}\D\H^{j_4-1}$. Lastly, if $j_3>0$, then $\alpha^{-1}(L(\pi))= \H^{j_1}\U \H^{j_2}\D \H^{m-2}\U \H^{j_3-1}\D \H^{j_4-m}$. 
\end{proof}

\section{Odd \& Even Grassmannian permutations}
\label{sec:evenoddPerms}

A permutation is said to be {\em even} if it has an even number of inversions (occurrences of the pattern 21); otherwise the permutation is said to be {\em odd}. In this section, we give some enumerative results concerning odd and even Grassmannian permutations.

\begin{thm} 
If $\Gr^{odd}_n$ is the set of odd permutations in $\Gr_n$, and $a_n=\abs{\Gr^{odd}_n}$, then
\begin{gather*}
 a_1 = 0, \; a_2=1, \text{ and} \\
 a_n = 2a_{n-2}+2^{n-2} \;\text{ for } n>2. 
\end{gather*}
\end{thm}
\begin{proof}
We start by proving the relations
\begin{equation}\label{eq:EvenOdd_split}
 a_{2m+1} = 2a_{2m} \;\text{ and }\; a_{2m+2} = a_{2m+1} + 2^{2m} \text{ for } m\ge 1. 
\end{equation}
For the first relation, we write $\Gr^{odd}_{2m+1} = \A \cup (\Gr^{odd}_{2m+1}\setminus \A)$, where $\A$ is the set of permutations in $\Gr^{odd}_{2m+1}$ that do not end with $2m+1$. Clearly, $\abs{\Gr^{odd}_{2m+1}\setminus \A} = \abs{\Gr^{odd}_{2m}}=a_{2m}$.

Consider the map $\xi:\Gr^{odd}_{2m} \to \A$ defined as follows. If $\pi\in \Gr^{odd}_{2m}$ does not end with $2m$, we let $\xi(\pi) = 1\oplus \pi$, which is in $\A$. Otherwise, if $\pi\in \Gr^{odd}_{2m}$ ends with $2m$, then we remove $2m$ from $\pi$, shift the remaining elements up by one, and insert the pair $(2m+1),1$ at the descent of $\pi$. For example,
\[ 351246 \to 35124 \to 46235 \to 46{\bf 71}235, \]
so $\xi(351246) = 4671235$. Observe that if the descent of $\pi$ is at position $j$, then the insertion of $1$ creates $j$ new inversions, and the insertion of $2m+1$ creates $2m-j$ new inversions. In other words, $\xi(\pi)$ has $2m$ more inversions than $\pi$ and has therefore the same parity.
The map $\xi$ is clearly bijective. Thus $\abs{\A}=\abs{\Gr^{odd}_{2m}}=a_{2m}$, and so $a_{2m+1}=2a_{2m}$.

In order to prove the second formula in \eqref{eq:EvenOdd_split}, we now consider the set $\E$ of permutations in $\Gr^{odd}_{2m+2}$ having their descent at even position. There is a bijection $\psi: \Gr^{odd}_{2m+1}\to \E$ defined as follows. If $\pi\in \Gr^{odd}_{2m+1}$ has its descent at even position, we let $\psi(\pi)=\pi\oplus 1$. Otherwise, if $\pi$ has its descent at odd position, $\psi(\pi)$ is the permutation obtained by inserting $2m+2$ at the descent of $\pi$. For example, $\psi(35124) = 35124{\bf 6}$ and $\psi(24513) = 245{\bf 6}13$. 

Thus $\abs{\E}=a_{2m+1}$. It remains to verify that the set $\Gr^{odd}_{2m+2}\setminus\E$ of permutations in $\Gr^{odd}_{2m+2}$, having their descent at odd position, has $2^{2m}$ elements. We proceed with a direct count.

Any permutation $\pi\in \Gr^{odd}_{2m+2}$ with descent at position $2k+1$ must be of the form
\[ \pi = i_1\cdots i_{2k+1} \,|\, i_{2k+2}\cdots i_{2m+2} \]
where $i_1,\dots, i_{2k+1}$, and  $i_{2k+2},\dots, i_{2m+2}$, are increasing sequences and $i_{2k+1}>i_{2k+2}$. The number of inversions of such a permutation is given by
\begin{align*} 
\inv(\pi) &= (i_1-1)+(i_2-2)+(i_{2k+1}-2k-1) \\
 &= (i_1+\cdots+i_{2k+1}) - (2k+1)(k+1) \\
 &\equiv  (i_1+\cdots+i_{2k+1}) - k - 1 \pmod 2,
\end{align*}
so $\pi$ is odd if and only if $i_1+\cdots+i_{2k+1}\equiv k \pmod 2$. Let
\[ \mathcal{D}_{\delta} = \Big\{ A\subset \{1,\dots,2m+2\}\!: \abs{A}=2k+1 \text{ and } \sum_{a\in A} a \equiv \delta\hspace{-1ex} \pmod 2\Big\}. \]
Clearly, $\abs{\mathcal{D}_0} + \abs{\mathcal{D}_1} = \binom{2m+2}{2k+1}$. Moreover, $\mathcal{D}_0\cong\mathcal{D}_1$ by means of the map 
\[ \{i_1,\dots,i_{2k+1}\} \mapsto  \{2m+3-i_1,\dots,2m+3-i_{2k+1}\}. \]
In other words, $\abs{\mathcal{D}_0}=\abs{\mathcal{D}_1}=\frac12 \binom{2m+2}{2k+1}$ gives the number of permutations in $\Gr^{odd}_{2m+2}$ having their descent at position $2k+1$. Therefore,
\[ \abs{\Gr^{odd}_{2m+2}\setminus\E} = \sum_{k=0}^{m} \frac12 \binom{2m+2}{2k+1} = 2^{2m}. \]
This finishes the proof of \eqref{eq:EvenOdd_split}. As a consequence, we have 
\[ a_{2m+2} = 2a_{2m} + 2^{2m} \;\text{ and }\; a_{2m+1} = 2(a_{2m-1} + 2^{2m-2}) = 2a_{2m-1} + 2^{2m-1}, \]
which combined give the claimed formula for $a_n=\abs{\Gr^{odd}_n}$.
\end{proof}

\begin{remark}
The sequence $a_n=\abs{\Gr^{odd}_n}$ starts with $0, 1, 2, 6, 12, 28, 56, 120, 240, 496,\dots,$ and it satisfies the recurrence relation $a_n=2a_{n-1}+2a_{n-2}-4a_{n-3}$ for $n>3$, cf.~\oeis{A122746}.
\end{remark}

\begin{cor}
If $b_n=\abs{\Gr^{even}_n}$, then
\begin{gather*}
 b_1 = 1, \; b_2=1, \text{ and} \\
 b_n = 2b_{n-2}+2^{n-2}+n-4 \;\text{ for } n>2. 
\end{gather*}
\end{cor}

\begin{remark}
From the previous equations, it follows that
\begin{equation*}
 \abs{\Gr^{odd}_n} = 2^{n-1} - 2^{\lfloor \frac{n-1}{2}\rfloor} \;\text{ and }\;
 \abs{\Gr^{even}_n} = 2^{n-1} + 2^{\lfloor \frac{n-1}{2}\rfloor} - n.
\end{equation*}
\end{remark}

\medskip
We conclude this section with a parity classification of Grassmannian Dyck paths.
\begin{prop}
The set $\Gr^{odd}_n$ is in bijection to the set of Grassmannian Dyck paths of semilength $n$ having an odd number of peaks at even height. Moreover, the elements of $\Gr^{even}_n$ correspond to Grassmannian Dyck paths with an even number of peaks at even height.
\end{prop}

\begin{proof}
We start with an example. Consider the odd Grassmannian permutation 
\[ \pi = 2\;3\;5\;7\;8\;[11]\;1\;4\;6\;9\;[10], \] 
which corresponds (by means of the map $\varphi$ from Section~\ref{sec:bijections}) to the Dyck path

\medskip
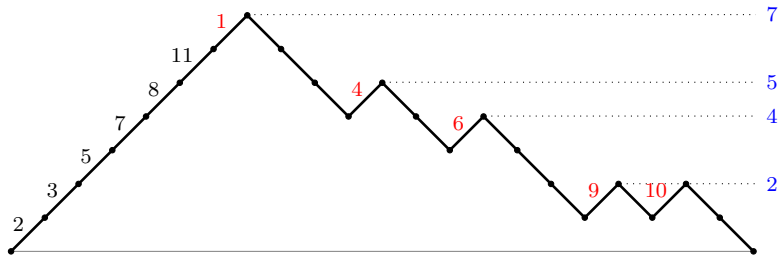
\begin{center}
\begin{tikzpicture}[scale=0.45]
\scriptsize
\dyckpath{0,0}{11}{1,1,1,1,1,1,1,0,0,0,1,0,0,1,0,0,0,1,0,1,0,0};
\foreach \x/\y in {7/7,11/5,14/4,18/2}{
	\draw[dotted] (\x,\y) -- (22,\y) node[right=1pt,blue] {\y};
};
\draw[red] (6.7,6.8) node[left=1pt] {1};
\draw[red] (10.7,4.8) node[left=1pt] {4};
\draw[red] (13.7,3.8) node[left=1pt] {6};
\draw[red] (17.7,1.8) node[left=1pt] {9};
\draw[red] (19.7,1.8) node[left=1pt] {10};
\draw (0.7,0.8) node[left=1pt] {2};
\draw (1.7,1.8) node[left=1pt] {3};
\draw (2.7,2.8) node[left=1pt] {5};
\draw (3.7,3.8) node[left=1pt] {7};
\draw (4.7,4.8) node[left=1pt] {8};
\draw (5.7,5.8) node[left=1pt] {11};
\end{tikzpicture}
\end{center}
Observe that $(i,j)$ is an inversion of $\pi$ if and only if the corresponding Dyck path $\varphi^{-1}(\pi)$ has the label $\pi(i)$ on its long ascent, and $\pi(j)$ is on a peak. Moreover, if a peak with label $\pi(j)$ is at height $h$, then $j$ is part of $h-1$ inversions. For instance, in the above path, label $1$ (height 7) contributes to six inversions, label $4$ (height 5) gives four inversions, label $6$ (height 4) gives three inversions, and labels $9$, $10$ (height 2) give one inversion each. 

In general, if a Grassmannian Dyck path $\varphi^{-1}(\pi)$ has $k$ peaks, say $p_1,\dots,p_k$, then the number of inversions of $\pi$ is given by
\[ \inv(\pi) = \sum_{i=1}^k (\textup{height}(p_i)-1) \equiv \#\{\text{peaks at even height}\} \!\!\pmod 2, \]
since peaks at an odd height contribute to an even number of inversions.
\end{proof}

\section*{Acknowledgement}
We thank Michael Weiner for many energizing and helpful discussions, and for carefully reading our manuscript. We are also very grateful to the Altoona Summer Undergraduate Research Fellowship that made it possible for Tomasko to dedicate herself to this project. 


\end{document}